\documentclass[amstex,10pt,reqn]{amsart}
\usepackage{fancyhdr}
\usepackage{geometry}
\usepackage{amsmath,amsfonts,amsthm,enumerate,float,mathtools,url,enumitem}
\usepackage[numbers]{natbib}
\newtheorem{lemma}{Lemma}[section]
\newtheorem{prop}{Proposition}
\newtheorem{thm}{Theorem}[section]

\newtheorem{definition}{Definition}[section]

\theoremstyle{remark}

\numberwithin{equation}{section}
\usepackage[utf8]{inputenc}
\usepackage{titlesec}
\titleformat{\section}[hang]
  {\normalfont\Large\bfseries}{\thesection}{1em}{}

\usepackage[T2A]{fontenc}
\usepackage[russian, english]{babel}
\usepackage{parskip}
\usepackage{amssymb}
\usepackage{microtype}
\begin{document}
\tolerance=1000  % Increase tolerance for overfull hboxes
\sloppy  % Allow more flexible line breaking
\title[]{Multiplicative and mining property for stability numbers of graphs.}
\maketitle
\begin{center}
\text{Metrose Metsidik and Lixiao Xiao} \vskip0.2in
College of Mathematical Sciences,\\ Xinjiang Normal University, Urumqi $830054$, P. R. China\\[0pt]
\textbf{E-mail:} \url{metrose@xjnu.edu.cn}.
\end{center}
\begin{abstract}
$f$-vertex stability number $vs_f(G)=\min\{|X|: X\subseteq V(G) \enspace \text{and} \enspace f(G-X)\neq f(G)\}$, and $f$-edge stability number is defined similarly by setting $X\subseteq E(G)$. In this paper, for multiplicative and mining invariant $f$, we give some general bounds for $f$-vertex/edge stability numbers of graphs and some results about the relations between the $f$-vertex/edge stability numbers of graphs and their components.
\vspace{1em}
\par
\noindent\textbf{Key words and phrases:} graph invariant, multiplicative property, mining property, vertex stability number, edge stability number.

\vspace{0.5cm}
\noindent\textbf{2020 Mathematics Subject Classi cation}:  05C15.
\end{abstract}
\maketitle
\section{Introduction}
Graph invariants are properties that remain constant under isomorphism. Rather than being tied to a specific representation of a graph, these properties generally pertain to the overall structure or characteristics of the graph. The exploration of graph invariants not only enhances our comprehension of the structural characteristics of graphs but also enables the development of general and efficient algorithms to address a wide range of problems in graph theory. As demonstrated in~\cite{ref-5,ref-3,ref-10,ref-9}, the stability of these invariants can be significantly influenced by modifications to the graph, such as the deletion of vertices or the addition/removal of edges. For instance, Haynes et al.~\cite{ref-10} have shown that altering the vertices and edges of a graph can directly impact specific invariants, including the minimum degree $\delta(G)$ and maximum degree $\Delta(G)$. 

Let $\mathcal {G}$ denote the family of simple finite graphs, and let an invariant $f$ be a function $f : \mathcal {G} \rightarrow \mathbb{R}^{+}$. Bauer et al.~\cite{ref-5} have investigated how changes in the number of vertices can affect the domination number, leading to the definition of the $f$-edge stability number. Kemnitz et al.~\cite{ref-8} have derived general bounds for the chromatic edge stability number and the chromatic bondage number. They further explored the $f$-edge chromatic stability number by setting $f=\chi^\prime(G)$ and provided bounds for the $f$-edge stability number~\cite{ref-2}. Akbari et al.~\cite{ref-4} have delved into the chromatic vertex stability number, proving that $vs_\chi(G)=ivs_\chi(G)$, where $ivs_\chi(G)$ represents the minimum number of independent vertices whose removal results in a decrease in the chromatic number $\chi(G)$.

Kemnitz and Marangio~\cite{ref-6,ref-7} have established upper and lower bounds for the edge total chromatic stability number and demonstrated a strong correlation between the chromatic number, the total chromatic number, and the $f$-edge stability number. They also provided a novel proof of the well-known Gallai's theorem using vertex/edge stability numbers~\cite{ref-11,ref-1}. Furthermore, for the disjoint union $H_1 \sqcup H_2$ of two graphs $H_1$ and $H_2$, Kemnitz and Marangio defined the additive property $f(H_1 \sqcup H_2)= f(H_1) + f(H_2)$ and the maxing property $f(H_1 \sqcup H_2)= \max \{f(H_1), f(H_2) \}$  for the invariant $f$. They showed that the difference and ratio between the edge chromatic stability number and the vertex chromatic stability number can be arbitrarily large, highlighting the intricate relationships between these graph invariants.

The additive and maxing properties of an invariant can be naturally extended to encompass multiplicative and mining properties of the invariant.
\begin{definition}
Let $H_1$ and $H_2$ be two disjoint graphs. An invariant $f$ is said to be multiplicative if it satisfies the property $f(H_1 \sqcup H_2)= f(H_1) \cdot f(H_2)$, and it is called mining if it satisfies the property $f(H_1 \sqcup H_2)= \min \{f(H_1), f(H_2) \}$.
\end{definition}

It is worth noting that many graph invariants adhere to either multiplicative or mining properties. For example, the number of spanning forests and the number of independent sets are multiplicative invariants, whereas the minimum degree and girth exhibit mining behavior.

In this paper, we investigate the $f$-vertex/edge stability numbers of finite simple graphs, focusing on invariants $f$ that exhibit multiplicative and mining properties.  In Section 2, leveraging the multiplicative and minimizing nature of $f$, we derive several upper bounds for the vertex stability number under specific conditions. We also explore generalizations of other invariants and establish relationships between the vertex stability number of a graph and its components. In Section 3, we extend our analysis to the edge stability number, presenting analogous results and further insights.

\section{Vertex stability number}
\begin{definition}
$f$-vertex stability number of a graph $G$ is defined as:
\begin{align*}
vs_f(G)=\begin{cases}\infty, \hskip 3cm\text{if } f(G-X)=f(G)\text{ for all }X\subseteq V(G);\\ \min \{|X|: X\subseteq V(G) \enspace \text{ and } \enspace f(G-X) \neq f(G)\}, \enspace\enspace \text{ otherwise }.\end{cases}
\end{align*}
\end{definition}

\begin{lemma}[\cite{ref-1}]\label{le-1}
$vs_f(G) \leq |X| + vs_f(G-X)$ for $X\subsetneq V(G)$.
\end{lemma}

Clearly, for any proper subset $X\subsetneq V(G)$, the subgraph $G-X$ is an induced subgraph of $G$. Consequently, Lemma \ref{le-1} can be rephrased as follows: for any induced subgraph $H$ of $G$, $vs_f(G)\leq vs_f(H)+|V(G)|-|V(H)|$.

Set $N(W)= \bigcup_{w \in W} N(w)$ for $W\subseteq V(G)$, where $N(w)=\{v\in V(G): vw\in E(G)\}$.

\begin{thm}\label{th-1}
Let $H$ be an induce graph of $G$ and $f$ be a  nonzero multiplicative invariant. Then $vs_f(G) \leq |X| + vs_f(H)$, where $X=N(V(H)) \setminus V(H)$.
\end{thm}
\begin{proof}
If $H=G$, the result holds trivially. By Lemma \ref{le-1}, we have $vs_f(G) \leq |X| + vs_f(G-X)$, so it suffices to show that $vs_f(G-X) \leq vs_f(H)$ when $H\neq G$. If $G-X=H$, the result is again trivial. Therefore, we assume that $G-X=H\sqcup H^{'}$, where $H$ and $H^{'}$ are two disjoint induced subgraphs of $G$ . If $vs_f(H)=\infty$, we are done. If $vs_f(H)<\infty$, then there exists a vertex subset $V^{'}\subseteq V(H)$ with $|V^{'}|= vs_f(H)$ such that $f(H) \neq f(H-V^{'})$. Since $f(G)$ is multiplicative and $f(H^{'})\neq0$,
\begin{align*}
f(G-X)&=f(H\sqcup H^{'})=f(H) \cdot f(H^{'})\\
&\neq f(H-V^{'}) \cdot f(H^{'})=f((H\sqcup H^{'})-V^{'})\\
&=f((G-X)-V^{'}),
\end{align*}
therefore $vs_f(G-X) \leq |V^{'}|= vs_f(H)$.
\end{proof}
\begin{thm}\label{th-2}
Let $f$ be a nonzero multiplicative invariant such that $f (K_1) \neq 1$, and let $G$ be a non-complete graph. Then $vs_f(G)\leq \delta(G)+1$.
\end{thm}
\begin{proof}
Let $w$ be a vertex of $G$ such that $d(w)= \delta(G)$. Set $X=N(w)$, $H=G[\{w\}]\cong K_1$, and $H^{'}=G-(X\cup \{w\})$, then $H^{'}\neq\emptyset$ and $G-X=H\sqcup H^{'}$. Since $f$ is multiplicative and $f(H^{'})\neq0$, we have
\begin{align*}
f(G-X)&=f(H\sqcup H^{'})= f(K_1)\cdot f(H^{'})\\
&\neq f(H^{'})=f((G-X)-w).
\end{align*}
Thus $vs_f(G-X)=1$, and therefore $vs_f(G)\leq {vs_f(G-X)+|X|} =1+\delta(G)$ by Lemma \ref{le-1}.
\end{proof}
Now we transfer Theorem \ref{th-1} to mining invariants.

\begin{thm}\label{th-3}
Let $X\subsetneq V(G)$ be a vertex set such that $G-X=H\sqcup H^{'}$, and $f$ be mining with $f(H)< f(H^{'})$. Then $vs_f(G)\leq {vs_f(H)+|X|}$.
\end{thm}
\begin{proof}
By Lemma \ref{le-1}, we have $vs_f(G)\leq {vs_f(G-X)+|X|}$. In the following, we aim to show that $vs_f(G-X)\leq vs_f(H)$. There is nothing to show for $vs_f(H)=\infty$. Let $vs_f(H)<\infty$, then there exists a vertex subset $S\subseteq V(H)$ with $|S|=vs_f(H)$ and $f(H) \neq f(H-S)$. Since $f$ is mining, we have
\begin{align*}
f(G-X)&=f(H\sqcup H^{'})=\min \{f(H), f(H^{'})\}=f(H)\\
&\neq  \min \{f(H-S), f(H^{'})\}= f((H-S)\sqcup H^{'})\\
&=f((H^{'}\sqcup H)-S)=f((G-X)-S),
\end{align*}
therefore $vs_f(G-X)\leq |S|=vs_f(H)$ as required.
\end{proof}
In the following, we generalize the above results about disjoint union of two graphs to disjoint union of arbitrarily many graphs.

\begin{thm}\label{th-4}
Let $G=H_1\sqcup \cdots \sqcup H_k$ and $f$ be a multiplicative invariant. If $f(H_j)\neq 0,1$ for $j=1, \cdots, k$, then $vs_f(G)= \min\{vs_f(H_j), |V(H_j)| : j=1, \cdots, k\}$.
\end{thm}
\begin{proof}
Let $N_j=\min \{vs_f(H_j), |V(H_j)|\}$ for $j=1,\cdots, k$.

If $vs_f(H_j)<\infty$, then there exists a vertex subset $X$ of $V(H_j)$ such that $|X|=vs_f(H_j)$ and $f(H_j-X)\neq f(H_j)$. Since $f$ is multiplicative and $f(H_i)\neq 0$ for $i=1, \cdots, k$,
\begin{align*}
f(G-X)&=f(H_1) \cdots f(H_{j-1})\cdot f(H_j-X)\cdot f(H_{j+1}) \cdots f(H_k)\\
&\neq f(H_1) \cdots f(H_{j-1})\cdot f(H_j)\cdot f(H_{j+1}) \cdots f(H_k)\\
&= f(G),
\end{align*}
we obtain $vs_f(G)\leq |X|=vs_f(H_j)$.

If $vs_f(H_j)=\infty$, then $N_j=|V(H_j)|$. Since $f(G)$ is multiplicative and $f(H_i)\neq 0,1$  for $i=1, \cdots, k$, we have
\begin{align*}
f(G-V(H_j))&=f(H_1 \sqcup \cdots \sqcup H_{j-1} \sqcup H_{j+1}\cup \cdots \sqcup H_k)\\
&= f(H_1) \cdots f(H_{j-1})\cdot f(H_{j+1}) \cdots f(H_k)\\
&\neq f(H_1) \cdots f(H_{j-1})\cdot f(H_j)\cdot f(H_{j+1}) \cdots f(H_k)\\
&= f(G),
\end{align*}
therefore $vs_f(G)\leq N_j$.

In the following, we prove $vs_f(G)\geq N_j$. Let $M$ be a vertex set of $V(G)$ and $|M|<N_j$. Then, by the definition of $N_j$, we obtain $f(H_j-M)= f(H_j)$ for all $j=1, \cdots ,k$. Since $f(G)$ is multiplicative, we have
\begin{align*}
f(G-M)&=f((H_1-M) \sqcup \cdots \sqcup (H_k-M))\\
&= f(H_1) \cdots f(H_k)= f(G).
\end{align*}
Thus $vs_f(G)\geq N_j$.
\end{proof}
\begin{thm}\label{th-5}
Let $G=H_1\sqcup \cdots \sqcup H_k$, $f$ be a multiplicative invariant, $f(H_i)=0$ for $i\in I$, and $vs_f(H_j)=\infty$ for $j\in J$. Then
\begin{align*}
vs_f(G)=\begin{cases}\infty,&J=\{1, \cdots, k\};\\ \sum _{i\in I} \min\{vs_f(H_i),|V(H_i)|\}, &I\neq \emptyset;\\ \min\{vs_f(H_i),|V(H_j)|: i\in \{1, \cdots, k\}\setminus J,&\\\hskip 4cm
j\in J \text{ and }f(H_j)\neq 1\},&\text{ otherwise }.\end{cases}
\end{align*}
\end{thm}
\begin{proof}
If $vs_f(H_j)=\infty$ for $i=1, \cdots, k$, then obviously $vs_f(G)=\infty$.

Let $T_i=V(H_i)$ for $i\in I$ and $vs_f(H_i)=\infty$. If $vs_f(H_i)<\infty$, then there exists a vertex subset $T_i$ of $V(H_i)$ with $|T_i|=vs_f(H_i)$ and $f(H_i-T_i)\neq f(H_i)$. Let $T=\bigcup_{i\in I}T_i$, then $|T|=\sum _{i\in I} \min\{vs_f(H_i),|V(H_i)|\}$. Since $f$ is multiplicative and $f(H_i)=0$ for $i\in I$, we have
\begin{align*}
f(G-T)&=f(\bigsqcup_{i\in I}(H_i-T_i)\sqcup\bigsqcup_{t\in \{1,\cdots,k\}\setminus I}(H_t))\\
&=\prod_{i\in I\setminus J}f(H_i-T_i)\cdot\prod_{t\in \{1,\cdots,k\}\setminus I}f(H_t)\\
&\neq 0=\prod_{i=1}^{k}f(H_i)=f(G).
\end{align*}
Thus $vs_f(G)\leq |T|$.

We will prove $vs_f(G)\geq |T|$. Let $N$ be a vertex subset of $V(G)$ and $|N|<|T|$. Then, by the definition of $T$, $f(H_i-N)=f(H_i)=0$ for some $i\in I$, and therefore $f(G-N)=0=f(G)$.

Let $J\neq \{1, \cdots, k\}$, $I=\emptyset$, $vs_f(H_i):=\min\{vs_f(H_i): i\in \{1, \cdots, k\}\setminus J\}$, and $|V(H_j)|:=\min\{|V(H_j)|: j\in J \text{ and }f(H_j)\neq 1\}$. Note that there exists a vertex subset $X_i$ of $V(H_i)$ with $|X_i|=vs_f(H_i)$ and $f(H_i-X_i)\neq f(H_i)$. Since $f(G)$ is multiplicative and $I=\emptyset$, we have
\begin{align*}
f(G-X_i)&=f(H_1\sqcup\cdots\sqcup H_{i-1}\sqcup(H_i-X_i)\sqcup H_{i+1}\sqcup \cdots\sqcup H_k)\\&=f(H_1)\cdots f(H_{i-1})f(H_i-X_i)f(H_{i+1})\cdots f(H_k)\\
&\neq f(H_1)\cdots f(H_{i-1})f(H_i)f(H_{i+1})\cdots f(H_k)=f(G),
\end{align*}
which implies $vs_f(G)\leq |X_i|=vs_f(H_i)$. If $f(H_j)\neq 1$ for some $j\in J$, then
\begin{align*}
f(G-V(H_j))&=f(H_1\sqcup\cdots\sqcup H_{j-1}\sqcup H_{j+1}\sqcup \cdots\sqcup H_k)\\&=f(H_1)\cdots f(H_{j-1})f(H_{j+1})\cdots f(H_k)\\
&\neq f(H_1)\cdots f(H_{j-1})f(H_j)f(H_{j+1})\cdots f(H_k)=f(G).
\end{align*}

Thus $vs_f(G)\leq \min\{|X_i|, |V(H_j)|\}$. In the following, we prove inverse inequality. Let $N^{'}$ be a vertex subset of $V(G)$ such that $|N^{'}|<\min\{|X_i|,$ $|V(H_j)|\}$. Then $f(H_i-N^{'})=f(H_i)$ for all $i\in \{1, \cdots, k\}$. Hence $f(G-N^{'})=f(G)$, and therefore $vs_f(G)\geq \min\{|X_i|, |V(H_j)|\}$.
\end{proof}

An invariant $f$ is called monotone increasing (resp. decreasing) if $f(G) \geq f(H)$ $(resp. f(G) \leq f(H) )$, where $H$ is an induce subgraph of $G$. We present the following two results for a multiplicative invariant that is both monotone and non-trivial.

\begin{thm}\label{th-6}
Let $f$ be a mining and monotone increasing invariant, $G=H_1\sqcup \cdots \sqcup H_k$, $f(H_i)=f(G)$ for $i\in I$, and $vs_f(H_j)=\infty$ for $j\in J$. Then
\begin{align*}
vs_f(G)=\begin{cases}\infty,&J=\{1, \cdots, k\};\\ \min\{\sum _{i\in I} |V(H_i)|,|X_t|: t\in\{1, \cdots, k\}\setminus I&\\\hskip3.5cm \text{ and }f(H_t-X_t)<f(G) \}, &I=J;\\ \min\{vs_f(H_i),|X_t|:i\in I\setminus J,\text{ and }\\\hskip2cm t\notin I\cup J\text{ and }f(H_t-X_t)<f(G)\},&otherwise.\end{cases}
\end{align*}
\end{thm}
\begin{proof}
If $vs_f(H_i)=\infty$ for $i=1, \cdots, k$, then $$f(G-X)=\min\{f(H_1-X), \cdots, f(H_k-X)\}=f(H_i-X)=f(G)$$ for all $X\subsetneq V(G)$, where $V(H_i)\nsubseteq X$, and therefore $vs_f(G)=\infty$.

Let $I=J\neq\{1, \cdots, k\}$. Then
\begin{align*}
f(G-\bigsqcup_{i\in I}V(H_i))&=f(\bigsqcup_{t\in \{1, \cdots, k\}\setminus I}H_t)\\
&= \min\{f(H_t): t\in \{1, \cdots, k\}\setminus I\}\\
&\neq f(G).
\end{align*}
Hence $vs_f(G)\leq \sum _{i\in I} |V(H_i)|$. By the definition of $f$-vertex stability number, $vs_f(G)\leq |X_t|$.

If $X\subset V(G)$ such that $|X|<\min\{\sum _{i\in I} |V(H_i)|,|X_t|\}$, then $V(H_i)\nsubseteq X$ for some $i\in I$ and $f(H_t-X)\geq f(G)$ for all $t\in \{1, \cdots, k\}\setminus I$. Therefore
\begin{align*}
f(G-X)&=f(H_i-X)=f(H_i)=f(G).
\end{align*}
Thus $vs_f(G)\geq \min\{\sum _{i\in I} |V(H_i)|,|X_t|\}$.

Notice that $I\neq\emptyset$ for mining $f$. Let $I\neq J$ and $i\in I\setminus J$. Then there exists a vertex subset $X_i$ of $V(H_i)$ with $|X_i|=vs_f(H_i)$ and $f(H_i-X_i)\neq f(H_i)$. Since $f$ is monotone increasing and $f(H_i)=f(G)$,
\begin{align*}
f(G-X_i)&=f(H_1\sqcup\cdots\sqcup(H_i-X_i)\cdots\sqcup H_k)\\
&= \min\{f(H_1), \cdots,f(H_i-X_i), \cdots, f(H_k)\}\\
&= f(H_i-X_i)<f(H_i)=f(G).
\end{align*}
Thus $vs_f(G)\leq\min\{vs_f(H_i): i\in I\setminus J\}$. Again by the definition of $f$-vertex stability number, $vs_f(G)\leq |X_t|$. It is easy to prove the reverse inequality.
\end{proof}
\begin{thm}\label{th-118}
Let $f$ be a mining and monotone decreasing invariant, $G=H_1\sqcup \cdots \sqcup H_k$, $f(H_i)=f(G)$ for $i\in I$, and $vs_f(H_j)=\infty$ for $j\in J$. Then
\begin{align*}
vs_f(G)=\begin{cases}\infty,&J=\{1, \cdots, k\};\\ \sum_{i\in I\setminus J}vs_f(H_i)+\sum_{i\in I\cap J}|V(H_i)|,&otherwise.\end{cases}
\end{align*}
\end{thm}
\begin{proof}
If $J=\{1, \cdots, k\}$, then again it is easy to prove that $vs_f(G)=\infty$.

There exists a vertex subset $X_i$ of $V(H_i)$ with $|X_i|=vs_f(H_i)$ and $f(H_i-X_i)\neq f(H_i)$ for all $i\in I\setminus J$. Let $X=\bigcup_{i\in I\setminus J}X_i\cup\bigcup_{i\in I\cap J}V(H_i)$. Since $f$ is monotone decreasing and $f(H_i)=f(G)$ for $i\in I$,
\begin{align*}
f(G-X)&=f(\bigsqcup_{i\in I\setminus J}(H_i-X_i)\sqcup\bigsqcup_{t\in \{1,\cdots,k\}\setminus I}H_t)\\
&= \min\{f(H_i-X_i),f(H_t):i\in I\setminus J\text{ and }t\in \{1,\cdots,k\}\setminus I\}\\
&>f(G).
\end{align*}
Thus $vs_f(G)\leq |X|$.

Now we prove $vs_f(G)\geq |X|$. Let $T$ be a vertex subset of $V(G)$ with $|T|< |X|$. By the selection of $|X|$, we have $f(H_i-T)=f(H_i)$ for some $i\in I\setminus J$ or $V(H_j)\nsubseteq T$ for  some $j\in I\cap J$. Then $$f(G-T)=f(H_i-T)= f(G)\text{ or } f(G-T)=f(H_j-T)= f(G),$$ and therefore the result follows.
\end{proof}

In the following, we investigate the vertex total chromatic stability number $vs_{\chi^{\prime\prime}}(G)$ of $G$. The proper $k$-vertex (resp. edge) coloring of $G$ is a function $c: V(G)$ (resp. $E(G))\rightarrow \{1, \cdots, k\}$ with $c(x)\neq c(y)$ for $xy\in E(G)$ (resp. adjacent edges $x$ and $y$). The vertex (resp. edge) chromatic number $\chi(G)$ (resp. $\chi^{'}(G)$) of $G$ is the minimum $k$ such that $G$ has a proper $k$-vertex (resp. $k$-edge) coloring.

The proper total $k$-coloring of $G$ is a function $c: V(G)\cup E(G)\rightarrow {1, \cdots, k}$ such that:
\begin{enumerate}[label=(\arabic*)]
    \item \label{it1} $G$ is properly $k$-vertex coloring.
    \item  \label{it2} $G$ is properly $k$-edge coloring.
    \item \label{it3} incident vertices and edges receive different colors.
   \end{enumerate}
The minimum number $k$ of the proper total $k$-coloring of $G$ is called proper total chromatic number $\chi^{\prime\prime}(G)$ of $G$. Note that $\Delta(G)+1\leq \chi^{\prime\prime}(G)$.

The well-known Total Coloring Conjecture \cite{ref-12} proposes that $$\chi^{\prime\prime}(G) \leq \Delta(G)+2$$ for any simple graph. Define $class(G)=\chi^{\prime\prime}(G)-\Delta(G)\ (\in \{1, 2\})$.

\begin{prop}\label{pr-1}
$vs_{\chi^{\prime\prime}}(G)\geq vs_\Delta(G)$ for $\chi^{\prime\prime}(G)= \Delta(G)+1$.
\end{prop}
\begin{proof}
It is obvious that $vs_{\chi^{\prime\prime}}(G)=vs_\Delta(G)=\infty$ for the empty graph $G$. So we assume that $E(G)\neq \emptyset$. Let $X$ be a vertex subset of $V(G)$ with $|X|=vs_{\chi^{\prime\prime}}(G)$ and $\chi^{\prime\prime}(G-X)\neq \chi^{\prime\prime}(G)$. Then
\begin{align*}
\Delta(G-X)+1&\leq \chi^{\prime\prime}(G-X)< \chi^{\prime\prime}(G)=\Delta(G)+1,
\end{align*}
and therefore $vs_\Delta(G) \leq |X|=vs_{\chi^{'}}(G)$.
\end{proof}
\begin{prop}\label{pr-2}
If the Total Coloring Conjecture is true and $\chi^{\prime\prime}(G)=\Delta(G)+2$, then $vs_{\chi^{\prime\prime}}(G)=\min\{vs_\Delta(G), vs_{class}(G)\}$.
\end{prop}
\begin{proof}
Note that $G$ is nonempty graph. Let $X$ be a vertex subset of $V(G)$ such that $|X|=vs_\Delta(G)$ and $\Delta(G-X)< \Delta(G)$. We have
\begin{align*}
\chi^{\prime\prime}(G-X)&\leq \Delta(G-X)+2< \Delta(G)+2=\chi^{\prime\prime}(G),
\end{align*}
therefore $vs_{\chi^{\prime\prime}}(G)\leq |X|=vs_\Delta(G)$. Let $Y$ be a vertex subset of $V(G)$ such that $|Y|=vs_{class}(G)$ and $class(G-Y)=1$. We obtain
\begin{align*}
\chi^{\prime\prime}(G-Y)&=\Delta(G-Y)+1<\Delta(G)+2=\chi^{\prime\prime}(G),
\end{align*}
and therefore $vs_{\chi^{\prime\prime}}(G)\leq |Y|=vs_{class}(G)$.

In the following, we prove $vs_{\chi^{\prime\prime}}(G)\geq \min\{vs_\Delta(G), vs_{class}(G)\}$. Let $T$ be a vertex subset of $V(G)$ such that $|T|=vs_{\chi^{\prime\prime}}(G)$ and  $\chi^{\prime\prime}(G-T)<\chi^{\prime\prime}(G)$. If $\Delta(G-T)=\Delta(G)$ and $class(G-T)=2$, then $$ \Delta(G)+2=\Delta(G-T)+2=\chi^{\prime\prime}(G-T)<\chi^{\prime\prime}(G)=\Delta(G)+2,$$ a contradiction. Therefore $\Delta(G-T)<\Delta(G)$ or $class(G-T)=1$. If $\Delta(G-T)<\Delta(G)$, then $|T|\geq vs_\Delta(G)$. If $class(G-T)=1$, then $class(G-T)<2=class(G)$, and therefore $|T|\geq vs_{class}(G)$.
\end{proof}

We can consider the value of $vs_{\chi^{'}}$, where $\chi^{'}$ is edge chromatic number. The edge chromatic number $\chi^{'}$ of a simple graph $G$ equals either $\Delta(G)$ or $\Delta(G)+1$ by Vizing's Theorem \cite{ref-13}. Let $class^{'}(G)=\chi^{'}(G)-\Delta(G)+1\ (\in \{1, 2\})$. Using similar methods of Propositions~\ref{pr-1}~and~\ref{pr-2}, we have the following two propositions.

\begin{prop}\label{pr-3}
$vs_{\chi^{'}}(G)\geq vs_\Delta(G)$ for $\chi^{'}(G)= \Delta(G)$.
\end{prop}

\begin{prop}\label{pr-14}
If $\chi^{'}(G)=\Delta(G)+1$, then $vs_{\chi^{'}}(G)=\min\{vs_\Delta(G), vs_{class^{'}}(G)\}$.
\end{prop}
\section{Edge stability number}
\begin{definition}
$f$-edge stability number of a graph $G$ is defined as:
\begin{align*}
es_f(G)=\begin{cases}\infty, \hskip 3cm\text{if } f(G-Y)=f(G)\text{ for all }Y\subseteq E(G);\\ \min \{|Y|: Y\subseteq E(G) \enspace \text{ and } \enspace f(G-Y) \neq f(G)\}, \enspace\enspace \text{ otherwise }.\end{cases}
\end{align*}    
\end{definition}

Notice that $es_f(G)= \infty$ if $G$ is an empty graph.

\begin{lemma}[\cite{ref-2}]\label{le-2}
Let $Y\subseteq E(G)$ and $H=G-Y$. Then $es_f(G)\leq {es_f(H)+|Y|}$.
\end{lemma}

Let $E_v$ be the set of edges incident to the vertex $v$, and $E(U,V)$ be the set of edges between the vertex sets $U$ and $V$.

\begin{thm}\label{th-7}
Let $f$ be multiplicative. If $u$ is a vertex of $G$ such that $f(G-u)> f(G)$ and $f(K_1)\geq 1$, or $f(G-u)=f(G)\neq0$ and $f(K_1)\neq1$, or $f(G-u)< f(G)$ and $f(K_1)\leq1$, then $es_f(G)\leq d(u)$.
\end{thm}
\begin{proof}
Since $f$ is multiplicative, we have $$f(G-E_u)=f((G-u)\sqcup K_1)=f(G-u)\cdot f(K_1).$$ If the conditions listed in the theorem are held, then $f(G-E_u)>f(G)$ or $f(G-E_u)<f(G)$, and therefore $es_f(G)\leq {|E_u|=d(u)}$.
\end{proof}
\begin{lemma}[\cite{ref-2}]\label{le-3}
Let $G$ be a graph and $H$ be a spanning subgraph of $G$. If $es_f(H)=1$, then $es_f(G)\leq {1+|E(G)|-|E(H)|}$.
\end{lemma}

\begin{thm}\label{th-8}
 Let $f$ be multiplicative. If $f(K_2)\neq f(2K_1)$, then $es_f(G)\leq \min\{d(x)+d(y)-1: xy\in E(G)\text{ and }f(G-\{x,y\})\neq0\}$.
\end{thm}
\begin{proof}
Let $H={G-E_x-E_y+xy}$ for $xy\in E(G)$, then $H={(G-\{x,y\})\sqcup K_2}$. Since $f$ is multiplicative and $f(K_2)\neq f(2K_1)$, we have
\begin{align*}
f(H)&=f(G-\{x,y\}) \cdot f(K_2)\\
&\neq f(G-\{x,y\}) \cdot f(2K_1)\\
&=f(H-xy)
\end{align*}
for $f(G-\{x,y\})\neq0$. This implies $es_f(H)=1$, and the result follows by Lemma \ref{le-3}.
\end{proof}
\begin{thm}\label{th-9}
Let $f$ be multiplicative and $H$ be a subgraph of $G$. Then $es_f(G)\leq es_f(H)+|E(V(H),V(G)\backslash V(H))|$ if $f(G-V(H))\neq0$.
\end{thm}
\begin{proof}
Let $T=(G-V(H))\sqcup H$, then $es_f(G)\leq es_f(T)+|E(V(H),V(G)\backslash V(H))|$ by Lemma \ref{le-2}.

If $es_f(H)=\infty$, then the result is obvious. If $es_f(H)<\infty$, then there exists an edge subset $Y$ of $E(H)$ such that $|Y|=es_f(H)$ and $f(H)\neq f(H-Y)$. Since $f$ is multiplicative, we have
\begin{align*}
f(T)&=f(G-V(H)) \cdot f(H) \\
&\neq f(G-V(H)) \cdot f(H-Y)\\
&= f((G-V(H))\cup (H-Y))\\
&=f(T-Y)
\end{align*}
for $f(G-V(H))\neq0$. Thus $es_f(T)\leq |Y|=es_f(H)$, and therefore the result follows.
\end{proof}
\begin{thm}\label{th-10}
Let $f$ be mining and $H$ be a subgraph of $G$. If $f(H)< f(G-V(H))$, then $es_f(G)\leq es_f(H)+|E(V(H),V(G)\backslash V(H))|$.
\end{thm}
\begin{proof}
Let $T=(G-V(H))\sqcup H$, then $$es_f(G)\leq es_f(T)+|E(V(H),V(G)\backslash V(H))|.$$

If $es_f(H)=\infty$, then the result is obvious. If $es_f(H)<\infty$, then there exists an edge subset $Y$ of $E(H)$ such that $|Y|=es_f(H)$ and $f(H)\neq f(H-Y)$. Since $f$ is mining and $f(H)< f(G-V(H))$, we have
\begin{align*}
f(T)&=\min\{f(G-V(H)), f(H)\}=f(H)\\
& \neq \min\{f(G-V(H)), f(H-Y)\}=f(T-Y).
\end{align*}
Thus $es_f(T)\leq |Y|=es_f(H)$, and  therefore the result follows.
\end{proof}

The above results show the influence of spanning subgraphs and induced subgraphs on the bound of $es_f(G)$. Note that $es_f(G)\leq es_f(H)+ |E(G)|-|E(H)|$ for a spanning subgraph $H$ of $G$, and $es_f(G)\leq es_f(H)+|E(V(H),V(G)\backslash V(H))|$ for a subgraph $H$ of $G$ with multiplicative $f$ and $f(G-V(H))\neq0$ or mining $f$ and $f(H)< f(G-V(H))$.

In the following, we generalize the above results to disjoint union of arbitrarily many graphs.

\begin{thm}\label{th-11}
Let $f$ be multiplicative, $G=H_1\sqcup \cdots \sqcup H_k$, $f(G)\neq0$, and $es_f(H_i)<\infty$ for $i=I$. Then
\begin{align*}
es_f(G)=\begin{cases}\infty,&I=\emptyset;\\ \min \{es_f(H_i): i\in I\},&otherwise.\end{cases}
\end{align*}
\end{thm}
\begin{proof}
If $I=\emptyset$, then $es_f(H_j)=\infty$ for all $j=1, \cdots, k$, and therefore $es_f(G)=\infty$.

Let $I\neq \emptyset$. Then there exists an edge subset $Y$ of $E(H_i)$ such that $|Y|=\min\{es_f(H_i): i\in I\}$ and $f(H_i-Y) \neq f(H_i)$ for some $i\in I$. Since $f$ is multiplicative and $f(G)\neq0$, we have
\begin{align*}
f(G-Y)&=f(H_1) \cdot f(H_2) \cdots f(H_i-Y) \cdots f(H_k)\\
&\neq f(H_1) \cdot f(H_2) \cdots f(H_i) \cdots f(H_k)=f(G),
\end{align*}
and therefore $es_f(G)\leq |Y|=\min\{es_f(H_i): i\in I\}$.

In the following, we prove $es_f(G)\geq |Y|$. Let $M$ be an edge subset of $E(G)$ with $|M|<|Y|$. By the selection of $Y$, we obtain $f(H_i-M)=f(H_i)$ for $i=1, \cdots ,k$. Thus $f(G-M)=f(G)$, then the result follows.
\end{proof}
\begin{thm}\label{th-12}
Let $f$ be mining and monotone decreasing, $G=H_1\sqcup \cdots \sqcup H_k$, $f(H_i)=f(G)$ for $i\in I$, and $es_f(H_j)=\infty$ for $j\in J$. Then
\begin{align*}
es_f(G)=\begin{cases}\infty,&I\cap J\neq\emptyset;\\ \sum_{i\in I} es_f(H_i),&otherwise.\end{cases}
\end{align*}
\end{thm}
\begin{proof}
Let $j\in I\cap J$ and $Y$ be an arbitrary edge subset of $E(G)$. Since $f$ is mining and monotone decreasing,
\begin{align*}
f(G-Y)&=f((H_1-Y)\sqcup \cdots\sqcup (H_j-Y)\sqcup\cdots \sqcup (H_k-Y))\\&=\min\{f(H_1-Y), \cdots, f(H_j-Y), \cdots, f(H_k-Y)\}\\
&= f(H_j-Y)=f(H_j)=f(G).
\end{align*}
Thus $es_f(G)=\infty$.

Now we assume that $I\cap J=\emptyset$. Since $f$ is mining, $I\neq\emptyset$. Then there there exists an edge subset $Y_i$ of $E(H_i)$ such that $|Y_i|=es_f(H_i)$ and $f(H_i-Y_i)\neq f(H_i)$ for $i\in I$. Set $Y=Y_1\cup \cdots Y_t$, then $|Y|=\sum_{i\in I}es_f(H_i)$. Since $f$ is mining, we have
\begin{align*}
f(G-Y)&= f(\bigsqcup_{i\in I}(H_i-Y_i)\sqcup\bigsqcup_{t\in \{1,\cdots,k\}\setminus I}H_t\\
&= \min \{f(H_i-Y_i), f(H_t):i\in I\text{ and }t\in \{1,\cdots,k\}\setminus I\}\\
&\neq \min\{f(H_i):i=1,\cdots,k\}=f(G),
\end{align*}
 and therefore $es_f(G)\leq |Y|$.

In the following, we prove $es_f(G)\geq |Y|$. Let $T$ be an edge subset of $E(G)$ with $|T|<|Y|$. By the definition of $Y$, we have $f(H_i-T)=f(H_i)$ for some $i\in I$. Since again $f$ is monotone decreasing, we have
\begin{align*}
f(G-T)&=f((H_1-T)\sqcup \cdots \sqcup (H_i-T)\sqcup \cdots \sqcup (H_k-T))\\
&= \min \{f(H_1-T), \cdots , f(H_i-T), \cdots, f(H_k-T)\}\\
&=f(H_i-T)=f(H_i)=f(G),
\end{align*}
and therefore the result follows.
\end{proof}
\begin{thm}\label{th-116}
Let $f$ be a mining and monotone increasing invariant, $G=H_1\sqcup \cdots \sqcup H_k$, $f(H_i)=f(G)$ for $i\in I$, and $es_f(H_j)=\infty$ for $j\in J$. Then
\begin{align*}
es_f(G)=\begin{cases}\infty,&J=\{1, \cdots, k\};\\ \min\{|Y_t|: t\notin I\text{ and }f(H_t-Y_t)<f(G) \}, &I=J;\\ \min\{es_f(H_i),|Y_t|:i\in I\setminus J\text{ and }\\\hskip2cm t\notin I\cup J\text{ and }f(H_t-Y_t)<f(G)\},&otherwise.\end{cases}
\end{align*}
\end{thm}
\begin{proof}
If $vs_f(H_i)=\infty$ for $i=1, \cdots, k$, then $$f(G-Y)=\min\{f(H_1-Y), \cdots, f(H_k-Y)\}=f(H_1-Y)=f(G)$$ for all $Y\subseteq E(G)$, and therefore $vs_f(G)=\infty$.

Let $I=J\neq\{1, \cdots, k\}$. Then, by the definition of $f$-edge stability number, $es_f(G)\leq \min\{|Y_t|: t\notin I\text{ and }f(H_t-Y_t)<f(G) \}$.

If $Y\subset E(G)$ such that $|Y|<\min\{|Y_t|: t\notin I\text{ and }f(H_t-Y_t)<f(G) \}$, then $f(H_t-Y)\geq f(G)$ for all $t\in \{1, \cdots, k\}\setminus I$. Therefore
\begin{align*}
f(G-Y)&=f(H_i-Y)=f(H_i)=f(G),
\end{align*}
where $i\in I$. Thus $es_f(G)\geq \min\{|Y_t|: t\notin I\text{ and }f(H_t-Y_t)<f(G) \}$.

Let $I\neq J$ and $i\in I\setminus J$. Then there exists an edge subset $Y_i$ of $E(H_i)$ with $|Y_i|=es_f(H_i)$ and $f(H_i-Y_i)\neq f(H_i)$. Since $f$ is monotone increasing and $f(H_i)=f(G)$,
\begin{align*}
f(G-Y_i)&=f(H_1\sqcup\cdots\sqcup(H_i-Y_i)\cdots\sqcup H_k)\\
&= \min\{f(H_1), \cdots,f(H_i-Y_i), \cdots, f(H_k)\}\\
&= f(H_i-Y_i)<f(H_i)=f(G).
\end{align*}
Thus $es_f(G)\leq\min\{es_f(H_i): i\in I\setminus J\}$. By the definition of $f$-edge stability number, $es_f(G)\leq \min\{|Y_t|: t\in\{1,\cdots,k\}\setminus(I\cup J)\text{ and }f(H_t-Y_t)<f(G) \}$. It is easy to prove that $$es_f(G)\geq\min\{es_f(H_i),|Y_t|:i\in I\setminus J,\text{ and }t\notin I\cup J\text{ and }f(H_t-Y_t)<f(G)\}.$$
\end{proof}
\begin{definition}
An edge subset $E$ of $G$ is covering if $E$ covers the edges of all nonempty spanning subgraphs $H$ with $f(H)=f(G)$. $\beta_f^{'}(G):=\min\{|E|: E \text{ is a covering set of } G\}$.
\end{definition}
\begin{lemma}[\cite{ref-11}]\label{le-4}
Let $f$ be monotone with respect to spanning subgraphs. If $es_f(G)< \infty$, then $es_f(G)=\beta_f^{'}(G)$.
\end{lemma}

In the following, we use the equation relationship between $es_f(G)$ and $\beta_f^{'}(G)$ in Lemma \ref{le-4} to discuss the lower bound of $es_f(G)$.

\begin{thm}\label{th-13}
 Let $f$ be monotone with respect to spanning subgraphs. If $H_1, \cdots, H_t$ be nonempty spanning subgraphs of $G$ with $f(H_i)=f(G)$ $(i=1, \cdots, t)$ such that $m$ edges appear in at least two $H_i$ and at most of $l\ (l\geq 1)$ $H_i$ are sharing an edge in common, then $es_f(G)\geq \frac{1}{l}\sum_{i=1}^t es_f(H_i)\geq t/l$ and $es_f(G)\geq \sum_{i=1}^t es_f(H_i)-m(l-1)$.
\end{thm}
\begin{proof}
There is noting to prove for $es_f(G)=\infty$. So we assume that $es_f(G)< \infty$. Since $f$ is monotonic with respect to spanning subgraphs, we have $es_f(G)=\beta_f^{'}(G)$ and $es_f(H_i)=\beta_f^{'}(H_i)< \infty\ (1\leq i\leq t)$ by Lemma \ref{le-4}.

Let $Y$ be an edge subset of $E(G)$ such that $|Y|=es_f(G)=\beta_f^{'}(G)$ and covers all nonempty spanning subgraphs $H_i$ of $G$ with $f(G)=f(H_i)$. Then, by the definition of $\beta_f^{'}(G)$, $Y$ contains at least $\beta_f^{'}(H_i)$ edges of $H_i$. Therefore, we have
\begin{align*}
\sum_{i=1}^{t}|Y\cap E(H_i)|&\geq \sum_{i=1}^t \beta_f^{'}(H_i)=\sum_{i=1}^t es_f(H_i)\geq t.
\end{align*}

Let $y=\min \{m, |Y|\}$, then $y\leq |Y|$ and $y \leq m$. Since at most $y$ edges of $Y\cap E(H_i)$ are calculated $l$ times and the remaining edges of $Y$ can be calculated at most once, we have
\begin{align*}
\sum_{i=1}^{t}|Y\cap E(H_i)|&\leq y \cdot l+(|Y|-y)=|Y|+y (l-1).
\end{align*}
Note that $y \leq |Y|$ and $\sum_{i=1}^{t}|Y\cap E(H_i)| \leq l\cdot |Y|$, and therefore
\begin{align*}
es_f(G)&=|Y|\geq \frac{1}{l}\sum_{i=1}^{t}|Y\cap E(H_i)|\geq \frac{1}{l}\sum_{i=1}^t es_f(H_i) \geq t/l.
\end{align*}
Similarly, we have
\begin{align*}
es_f(G)&=|Y|\geq \sum_{i=1}^{t}|Y\cap E(H_i)|-m(l-1)\\
&\geq \sum_{i=1}^t es_f(H_i)-m(l-1).
\end{align*}
\end{proof}

\noindent{\bf Acknowledgments} This work is supported by Natural Science Foundation of Xinjiang (Grant/Award Numbers: 2024D01A89) and National Natural Science Foundation of China (Grant/Award Numbers: 11961070).

\end{document}